\documentclass[12pt]{amsart}

\usepackage{graphicx,amsfonts}
\usepackage{latexsym}
\usepackage{amsmath}
\usepackage{amsmath}
\usepackage{color}
\usepackage{txfonts}
\usepackage[colorlinks]{hyperref}
\usepackage{mathptmx}
\usepackage{mathpazo}
\usepackage{pifont}
\usepackage{graphicx}
\usepackage{amssymb}
\usepackage{tikz}
\usetikzlibrary{positioning}

\newtheorem{theorem}{Theorem}
\newtheorem{proposition}{Proposition}
\newtheorem{lemma}{Lemma}
\newtheorem{corollary}{Corollary}

\newtheorem{remark}{Remark}

\begin{document}

\markboth{G. Guettai, D. Laissaoui, M. Rahmani and M. Sebaoui} {On poly-Bell numbers and polynomials}

%
%

\title{On poly-Bell numbers and polynomials}

\author{Ghania Guettai}

\address{USTHB, Faculty of Mathematics\\BP 32, El Alia, 16111, Bab Ezzouar\\ Algiers, Algeria}
\email{guettai78@yahoo.fr}

\author{Diffalah Laissaoui}

\address{Dr Yahia Far\`{e}s, University of M\'{e}d\'{e}a\\
 26000 M\'{e}d\'{e}a \\
Algeria\\}
\email{laissaoui.diffalah74@gmail.com}

\author{Mourad Rahmani}

\address{USTHB, Faculty of Mathematics\\BP 32, El Alia, 16111, Bab Ezzouar\\ Algiers, Algeria}
\email{mrahmani@usthb.dz, rahmani.mourad@gmail.com}

\author{Madjid Sebaoui}

\address{Dr Yahia Far\`{e}s, University of M\'{e}d\'{e}a\\
 26000 M\'{e}d\'{e}a \\
Algeria\\}
\email{msebaoui@gmail.com}

\maketitle

\begin{abstract}
This paper aims to construct a new family of numbers and
polynomials which are related to the Bell numbers and polynomials by means of
the confluent hypergeometric function. We give various properties of these
numbers and polynomials (generating functions, explicit formulas, integral
representations, recurrence relations, probabilistic representation,...). We
also derive some combinatorial sums including the generalized Bernoulli
polynomials, lower incomplete gamma function, generalized Bell polynomials.
Finally, by applying Cauchy formula for repeated integration, we introduce
poly-Bell numbers and polynomials.\\[5pt]

\end{abstract}


\section{Introduction }

Recently, Rahmani in \cite{Rahmani2015} published a paper on the function
\[
g_{p}\left(  z\right)  =\left(  p+1\right)  !%
{\displaystyle\sum\limits_{n\geq0}}
\frac{n!}{\left(  n+p+1\right)  !}\left(  1-e^{z}\right)  ^{n},
\]
which generalized the generating function of Bernoulli numbers and gave an
interesting algorithm for computing Bernoulli numbers and polynomials. Our
main goal in this paper is to consider a class of numbers outcome from the
following generating function which generalize Bell numbers
\[
f_{p}\left(  z\right)  =p!%
{\displaystyle\sum\limits_{n\geq0}}
\frac{\left(  e^{z}-1\right)  ^{n}}{\left(  n+p\right)  !}.
\]
and we give some of its properties.

The present paper is organized as follows. We first introduce in Section $2$,
our notations and definitions. Then we present in Section $3$ some properties
related to the $p$-Bell numbers. The poly-Bell numbers is introduced in
Section $4$. The $p$-Bell and poly-Bell polynomials are presented in Section
$5$. Finally, the probabilistic representation of $p$-Bell polynomials is
given in Section $6.$

\section{Preliminaries}

As usual \cite {Knuth}, the falling factorial $x^{\underline{n}}$ $\left(
x\in\mathbb{C}\right)  $ is defined by
\[
\left\{
\begin{tabular}
[c]{ll}%
$x^{\underline{0}}=1,$ & \\
$x^{\underline{n}}=x\left(  x-1\right)  \cdots\left(  x-n+1\right),  $ & for
$n>0$%
\end{tabular}
\right.
\]
and the rising factorial denoted by $x^{\overline{n}}$, is defined by $x^{\overline{n}%
}=x\left(  x+1\right)  \cdots\left(  x+n-1\right)  $ with $x^{\overline{0}}%
=1$. The (signed) Stirling numbers of the first kind $s\left(  n,k\right)  $
are the coefficients in the expansion%
\[
x^{\underline{n}}=%
{\displaystyle\sum\limits_{k=0}^{n}}
s\left(  n,k\right)  x^{k},
\]
and satisfy the recurrence relation given by%
\begin{equation}
s\left(  n+1,k\right)  =s\left(  n,k-1\right)  -ns\left(  n,k\right)  \text{
\ }\left(  1\leq k\leq n\right)  .\label{DT4}%
\end{equation}
The Stirling numbers of the second kind, denoted $%
\genfrac{\{}{\}}{0pt}{}{n}{k}%
$ are the coefficients in the expansion
\[
x^{n}=%
{\displaystyle\sum\limits_{k=0}^{n}}
\genfrac{\{}{\}}{0pt}{}{n}{k}%
x^{\underline{k}}.
\]
The Stirling numbers of the second kind $%
\genfrac{\{}{\}}{0pt}{}{n}{k}%
$ count the number of ways to partition a set of $n$ elements into exactly $k$
nonempty subsets. The number of all partitions is the Bell number $\phi_{n},$ thus%
\[
\phi_{n}=%
{\displaystyle\sum\limits_{k=0}^{n}}
\genfrac{\{}{\}}{0pt}{}{n}{k}%
.
\]
The polynomials
\begin{equation}
\phi_{n}\left(  x\right)  =%
{\displaystyle\sum\limits_{k=0}^{n}}
\genfrac{\{}{\}}{0pt}{}{n}{k}%
x^{k} \label{Aich3}%
\end{equation}
are called single-variable Bell polynomials or exponential polynomials.

The exponential generating functions are respectively%
\begin{align}%
{\displaystyle\sum\limits_{n\geq k}}
s\left(  n,k\right)  \frac{z^{n}}{n!}  &  =\frac{1}{k!}\left(  \ln\left(
1+z\right)  \right)  ^{k},\label{firstkind}\\%
{\displaystyle\sum\limits_{n\geq k}}
\genfrac{\{}{\}}{0pt}{}{n}{k}%
\frac{z^{n}}{n!}  &  =\frac{1}{k!}\left(  e^{z}-1\right)  ^{k}
\label{secondkind}%
\end{align}
and%
\begin{equation}%
{\displaystyle\sum\limits_{n\geq0}}
\phi_{n}\left(  x\right)  \frac{z^{n}}{n!}=\exp\left(  x\left(  e^{z}%
-1\right)  \right)  . \label{Bellgen}%
\end{equation}

\section{The $p$-Bell numbers}

In this section, we introduce and study a new generalization of the Bell
number which we call the $p$-Bell numbers. For every integer $p\geq0$, we
define a sequence of rational numbers $\mathcal{B}_{n,p}$ $\left(
n\geq0\right)  $ by%
\begin{equation}
f_{p}\left(  z\right)  :=%
{\displaystyle\sum\limits_{n\geq0}}
\mathcal{B}_{n,p}\frac{z^{n}}{n!}=\text{ }_{1}F_{1}\left(
\begin{array}
[c]{c}%
1\\
p+1
\end{array}
;e^{z}-1\right)  =%
{\displaystyle\sum\limits_{n\geq0}}
\binom{n+p}{p}^{-1}\frac{\left(  e^{z}-1\right)  ^{n}}{n!}, \label{F1}%
\end{equation}
where $\mathcal{B}_{n,0}:=\phi_{n}$ denotes the classical Bell numbers and
$_{1}F_{1}\left(
\begin{array}
[c]{c}%
a\\
c
\end{array}
;z\right)  $ denotes the Kummer confluent hypergeometric function \cite{Askey} is defined by
\begin{equation}%
{\displaystyle\sum\limits_{n\geq0}}
\frac{a^{\overline{n}}}{c^{\overline{n}}}\frac{z^{n}}{n!}. \label{R5}%
\end{equation}
The first three exponential generating functions are:
\begin{align}%
{\displaystyle\sum\limits_{n\geq0}}
\mathcal{B}_{n,1}\frac{z^{n}}{n!}  &  =\frac{\exp\left(  e^{z}-1\right)
-1}{e^{z}-1},\label{BN2}\\%
{\displaystyle\sum\limits_{n\geq0}}
\mathcal{B}_{n,2}\frac{z^{n}}{n!}  &  =\frac{2\left(  \exp\left(
e^{z}-1\right)  -e^{z}\right)  }{\left(  e^{z}-1\right)  ^{2}},\nonumber\\%
{\displaystyle\sum\limits_{n\geq0}}
\mathcal{B}_{n,3}\frac{z^{n}}{n!}  &  =\text{ }\frac{3\left(  2\exp\left(
e^{z}-1\right)  -e^{2z}-1\right)  }{\left(  e^{z}-1\right)  ^{3}}.\nonumber
\end{align}

\subsection{Generating functions and Explicit formulas}

We start with an explicit formula for $p$-Bell numbers $\mathcal{B}_{n,p}$
involving Stirling numbers of the second kind.
\begin{theorem}
For $p\geq0$, we have
\begin{equation}
\mathcal{B}_{n,p}=%
{\displaystyle\sum\limits_{k=0}^{n}}
\binom{k+p}{k}^{-1}%
\genfrac{\{}{\}}{0pt}{}{n}{k}%
. \label{Exp11}%
\end{equation}
\end{theorem}
\begin{proof}
By using (\ref{secondkind}), we obtain
\begin{align*}%
{\displaystyle\sum\limits_{n\geq0}}
{\displaystyle\sum\limits_{k=0}^{n}}
\binom{k+p}{k}^{-1}%
\genfrac{\{}{\}}{0pt}{}{n}{k}%
\frac{z^{n}}{n!}  &  =%
{\displaystyle\sum\limits_{k\geq0}}
\binom{k+p}{k}^{-1}%
{\displaystyle\sum\limits_{n\geq k}}
\genfrac{\{}{\}}{0pt}{}{n}{k}%
\frac{z^{n}}{n!}\\
&  =%
{\displaystyle\sum\limits_{k\geq0}}
\frac{p!k!}{\left(  k+p\right)  !}\frac{\left(  e^{z}-1\right)  ^{k}}{k!}\\
&  =%
{\displaystyle\sum\limits_{k\geq0}}
\frac{\left(  1\right)  ^{\overline{k}}}{\left(  p+1\right)  ^{\overline{k}}%
}\frac{\left(  e^{z}-1\right)  ^{k}}{k!}\\
&  =\text{ }_{1}F_{1}\left(
\begin{array}
[c]{c}%
1\\
p+1
\end{array}
;e^{z}-1\right) \\
&  =%
{\displaystyle\sum\limits_{n\geq0}}
\mathcal{B}_{n,p}\frac{z^{n}}{n!}.
\end{align*}
Comparing the coefficients of $\frac{z^{n}}{n!}$, on both sides, we arrive at
the result (\ref{Exp11}).
\end{proof}
In particular, for $p=1$ we have%
\begin{align*}
\mathcal{B}_{n,1}  &  =%
{\displaystyle\sum\limits_{k=0}^{n}}
\frac{1}{k+1}%
\genfrac{\{}{\}}{0pt}{}{n}{k}%
\\
&  =%
{\displaystyle\int\limits_{0}^{1}}
\phi_{n}\left(  t\right)  dt
\end{align*}
and we can write $\mathcal{B}_{n,1}$ also in the following form.
\begin{proposition}
\bigskip For $n\geq0$, we have%
\begin{equation}
\mathcal{B}_{n,1}=%
{\displaystyle\sum\limits_{k=0}^{n}}
\dbinom{n}{k}\frac{\phi_{k+1}B_{n-k}}{k+1}, \label{Ram}%
\end{equation}
where $B_{n}$ denotes the $n$-th Bernoulli number \cite{Knuth}, which is defined by means
of the following generating function
\[%
{\displaystyle\sum\limits_{n\geq0}}
B_{n}\frac{z^{n}}{n!}=\frac{z}{e^{z}-1}.
\]
\end{proposition}
\begin{proof}
We can rewrite $\left(  \ref{BN2}\right)  $ as
\begin{align*}%
{\displaystyle\sum\limits_{n\geq0}}
\mathcal{B}_{n,1}\frac{z^{n}}{n!}  &  =\frac{z}{e^{z}-1}\left(  \frac
{\exp\left(  e^{z}-1\right)  -1}{z}\right) \\
&  =%
{\displaystyle\sum\limits_{n\geq0}}
\left(
{\displaystyle\sum\limits_{k=0}^{n}}
\dbinom{n}{k}\frac{\phi_{k+1}B_{n-k}}{k+1}\right)  \frac{z^{n}}{n!}.
\end{align*}
Equating the coefficients of $\frac{z^{n}}{n!}$, we get the desired result.
\end{proof}
\begin{remark}
We note that the formula (\ref{Ram}) is a particular case of Ramanujan's
identity \cite[Example 4, p. 51]{Berndt}.
\end{remark}
\begin{theorem}
The double exponential generating function of $\mathcal{B}_{n,p}$ is given by%
\[%
{\displaystyle\sum\limits_{p\geq0}}
{\displaystyle\sum\limits_{n\geq p}}
\mathcal{B}_{n,p}\frac{z^{n}}{n!}\frac{y^{p}}{p!}=\frac{\left(  e^{z}%
-1\right)  \exp\left(  e^{z}-1\right)  }{e^{z}-1-y}.
\]
\end{theorem}
\begin{proof}%
\begin{align*}%
{\displaystyle\sum\limits_{p\geq0}}
{\displaystyle\sum\limits_{n\geq p}}
\mathcal{B}_{n,p}\frac{z^{n}}{n!}\frac{y^{p}}{p!}  &  =%
{\displaystyle\sum\limits_{p\geq0}}
{\displaystyle\sum\limits_{n\geq0}}
\frac{\left(  e^{z}-1\right)  ^{n}}{\left(  n+p\right)  !}y^{p}\\
&  =%
{\displaystyle\sum\limits_{n\geq0}}
\frac{\left(  e^{z}-1\right)  ^{n}}{n!}%
{\displaystyle\sum\limits_{p\geq0}}
\left(  \frac{y}{e^{z}-1}\right)  ^{p}\\
&  =\frac{\left(  e^{z}-1\right)  \exp\left(  e^{z}-1\right)  }{e^{z}-1-y}.
\end{align*}
\end{proof}
In order to establish some properties of $\mathcal{B}_{n,p}$, recall that the
$r$-Stirling numbers of the second kind \cite{Broder} $%
\genfrac{\{}{\}}{0pt}{}{n}{k}_{r}$ counts the number of partitions of a set of $n$ objects into exactly k
nonempty, disjoint subsets, such that the first $r$ elements are in distinct
subsets. The exponential generating function is given by%
\[
{\displaystyle\sum\limits_{n\geq k}}%
\genfrac{\{}{\}}{0pt}{}{n+r}{k+r}%
_{r}\frac{z^{n}}{n!}=\frac{1}{k!}e^{rz}\left(  e^{z}-1\right)  ^{k}.
\]
Now, by means of the generalized Stirling transform \cite{Rahmani2014}, we obtain
the following result.
\begin{theorem}
\label{DA2}For $n,m\geq0$, we have
\[
{\sum\limits_{k=0}^{m}}s\left(  m,k\right)  \mathcal{B}_{n+k,p}={\sum
\limits_{k=0}^{n}}%
\genfrac{\{}{\}}{0pt}{}{n+m}{k+m}%
_{m}\dbinom{m+k+p}{p}^{-1}.
\]
\end{theorem}
In the next, we give some integral representations involving the
exponential polynomials. It follows from the general theory of hypergeometric
functions that the confluent hypergeometric function has an integral
representation \cite{Askey}
\begin{equation}
_{1}F_{1}\left(
\begin{array}
[c]{c}%
a\\
b
\end{array}
;z\right)  =\frac{\Gamma\left(  b\right)  }{\Gamma\left(  b-a\right)
\Gamma\left(  a\right)  }%
{\displaystyle\int\limits_{0}^{1}}
e^{zt}t^{a-1}\left(  1-t\right)  ^{b-a-1}dt, \label{R1}%
\end{equation}
where $\Gamma$ denotes the gamma function.
\begin{theorem}
For $p\geq1$, we have%
\begin{equation}%
{\displaystyle\sum\limits_{n\geq0}}
\mathcal{B}_{n,p}\frac{z^{n}}{n!}=p%
{\displaystyle\int\limits_{0}^{1}}
\exp(\left(  e^{z}-1\right)  t)\left(  1-t\right)  ^{p-1}dt. \label{R2}%
\end{equation}
\end{theorem}
\begin{corollary}
For $p\geq1$, we have%
\begin{equation}
\mathcal{B}_{n,p}=p%
{\displaystyle\int\limits_{0}^{1}}
\left(  1-t\right)  ^{p-1}\phi_{n}\left(  t\right)  dt. \label{za}%
\end{equation}
\end{corollary}
\begin{proof}
We have%
\begin{align*}%
{\displaystyle\sum\limits_{n\geq0}}
\mathcal{B}_{n,p}\frac{z^{n}}{n!}  &  =p%
{\displaystyle\int\limits_{0}^{1}}
{\displaystyle\sum\limits_{n\geq0}}
\phi_{n}\left(  t\right)  \frac{z^{n}}{n!}\left(  1-t\right)  ^{p-1}dt\\
&  =%
{\displaystyle\sum\limits_{n\geq0}}
\left(  p%
{\displaystyle\int\limits_{0}^{1}}
\phi_{n}\left(  t\right)  \left(  1-t\right)  ^{p-1}dt\right)  \frac{z^{n}%
}{n!}.
\end{align*}
Comparing the coefficients of $\frac{z^{n}}{n!}$, on both sides, we arrive at
the result (\ref{za}).
\end{proof}
The integral representation $\left(  \ref{za}\right)  $ can be also proved by
using a well-known result \cite{Sury}
\[
\dbinom{n}{k}^{-1}=\left(  n+1\right)
{\displaystyle\int\limits_{0}^{1}}
x^{k}\left(  1-x\right)  ^{n-k}dx
\]
and $\left(  \ref{Exp11}\right)  $, we write%
\begin{align*}
\mathcal{B}_{n,p}  &  =%
{\displaystyle\sum\limits_{k=0}^{n}}
\binom{k+p}{k}^{-1}%
\genfrac{\{}{\}}{0pt}{}{n}{k}%
\\
&  =%
{\displaystyle\int\limits_{0}^{1}}
\left(  1-x\right)  ^{p}%
{\displaystyle\sum\limits_{k=0}^{n}}
\left(  k+p+1\right)
\genfrac{\{}{\}}{0pt}{}{n}{k}%
x^{k}dx\\
&  =%
{\displaystyle\int\limits_{0}^{1}}
\left(  1-x\right)  ^{p}\left(  \left(  p+1\right)  \phi_{n}\left(  x\right)
+x\frac{d}{dx}\phi_{n}\left(  x\right)  \right)  dx\\
&  =\left(  p+1\right)
{\displaystyle\int\limits_{0}^{1}}
\left(  1-x\right)  ^{p}\phi_{n}\left(  x\right)  dx+%
{\displaystyle\int\limits_{0}^{1}}
\left(  1-x\right)  ^{p}x\left(  \frac{d}{dx}\phi_{n}\left(  x\right)
\right)  dx\\
&  =p%
{\displaystyle\int\limits_{0}^{1}}
\left(  1-x\right)  ^{p}\phi_{n}\left(  x\right)  dx+p%
{\displaystyle\int\limits_{0}^{1}}
\left(  1-x\right)  ^{p-1}x\phi_{n}\left(  x\right)  dx\\
&  =p%
{\displaystyle\int\limits_{0}^{1}}
\left(  1-x\right)  ^{p-1}\phi_{n}\left(  x\right)  dx.
\end{align*}
Now, using the well-known Dobinski's formula for the Bell polynomials \cite{Comtet}
\begin{equation}
\phi_{n}\left(  x\right)  =e^{-x}\sum\limits_{k\geq0}\frac{k^{n}}{k!}x^{k},
\label{zaa}%
\end{equation}
we get, the Dobinski's formula for the $p$-Bell numbers.
\begin{corollary}
The $p$-Bell numbers $\mathcal{B}_{n,p}$ can also be written as
\begin{equation}
\mathcal{B}_{n,p}=\sum\limits_{k\geq0}\binom{p+k}{k}^{-1}\text{ }_{1}%
F_{1}\left(
\begin{array}
[c]{c}%
k+1\\
p+k+1
\end{array}
;-1\right)  \frac{k^{n}}{k!}. \label{Doni}%
\end{equation}
\end{corollary}
\begin{proof}
From $\left(  \ref{za}\right)  $ and $\left(  \ref{zaa}\right)  ,$ we obtain%
\[
\mathcal{B}_{n,p}=p\sum\limits_{k\geq0}\frac{k^{n}}{k!}%
{\displaystyle\int\limits_{0}^{1}}
e^{-t}\left(  1-t\right)  ^{p-1}t^{k}dt.
\]
and (\ref{Doni}) follows from (\ref{R1}).
\end{proof}
We present here another expression for the generating function of
$\mathcal{B}_{n,p}$ involving the lower incomplete gamma function. Recall that
the lower incomplete gamma function $\gamma\left(  s,z\right)  $ is defined
as:%
\[
\gamma\left(  s,z\right)  =%
{\displaystyle\int\limits_{0}^{z}}
e^{-t}t^{s-1}dt.
\]
\begin{theorem}
For $p\geq1$, the exponential generating function for $p$-Bell numbers is given
by%
\[%
{\displaystyle\sum\limits_{n\geq0}}
\mathcal{B}_{n,p}\frac{z^{n}}{n!}=\frac{p\left(  \exp(e^{z}-1)\right)
}{\left(  e^{z}-1\right)  ^{p}}\gamma\left(  p,e^{z}-1\right)  .
\]
\end{theorem}
\begin{proof}
This can be achieved by a simple change of variable $t=1-y/(e^{z}-1)$ in
(\ref{R2}).
\end{proof}
\begin{theorem}
The exponential generating function for $p$-Bell numbers is given by%
\begin{equation}%
{\displaystyle\sum\limits_{n\geq0}}
\mathcal{B}_{n,p}\frac{z^{n}}{n!}=\frac{p!\left(  \exp(e^{z}-1)\right)
}{\left(  e^{z}-1\right)  ^{p}}-%
{\displaystyle\sum\limits_{k=1}^{p}}
\frac{p^{\underline{k}}}{\left(  e^{z}-1\right)  ^{k}} \label{hb}.
\end{equation}
\end{theorem}
\begin{proof}
After an integration by parts, we can rewrite (\ref{R2}) as
\begin{equation}
f_{p}\left(  z\right)  =\frac{p}{e^{z}-1}f_{p-1}\left(  z\right)  -\frac
{p}{e^{z}-1}. \label{re}%
\end{equation}
Now, applying (\ref{re}) inductively we get the desired result.
\end{proof}
Next, we will show the relationship between the $p$-Bell numbers, Bell numbers
and the generalized Bernoulli numbers. Recall that the generalized Bernoulli
polynomials $B_{n}^{\left(  \alpha\right)  }(x)$ of degree $n$ in $x$ are
defined by the exponential generating function \cite{Boutiche, Srivastava2}
\begin{equation}
\left(  \frac{t}{e^{t}-1}\right)  ^{\alpha}e^{xt}={\displaystyle\sum
\limits_{n=0}^{\infty}}B_{n}^{\left(  \alpha\right)  }(x)\;\frac{t^{n}}%
{n!}\text{ \ \ }\left(  \left\vert t\right\vert <2\pi;\;1^{\alpha}:=1\right)
\label{GBernoulli}%
\end{equation}
for arbitrary parameter $\alpha$. In particular $B_{n}^{\left(  \alpha\right)
}(0):=B_{n}^{\left(  \alpha\right)  }$ denotes the generalized Bernoulli
numbers of order $\alpha$, and $B_{n}^{\left(  1\right)  }(x)$ denotes the
classical Bernoulli polynomials.
\begin{corollary}
For $n,p\geq0$, we have%
\begin{equation} \label{A22}
\mathcal{B}_{n,p}=%
{\displaystyle\sum\limits_{k=0}^{n+p}}
\binom{n+p}{k}\binom{n+p}{p}^{-1}\phi_{n+p-k}B_{k}^{\left(  p\right)  }-%
{\displaystyle\sum\limits_{k=1}^{p}}
\binom{n+k}{k}^{-1}\binom{p}{k}B_{n+k}^{\left(  k\right)  }.
\end{equation}
\end{corollary}
\begin{proof}
From (\ref{hb}), we write%
\begin{align*}%
{\displaystyle\sum\limits_{n\geq0}}
\mathcal{B}_{n,p}\frac{z^{n}}{n!}  &  =p!%
{\displaystyle\sum\limits_{n\geq0}}
\left(
{\displaystyle\sum\limits_{k=0}^{n}}
\binom{n}{k}\phi_{n-k}B_{k}^{\left(  p\right)  }\right)  \frac{z^{n-p}}{n!}-%
{\displaystyle\sum\limits_{n\geq0}}
{\displaystyle\sum\limits_{k=1}^{p}}
p^{\underline{k}}B_{n}^{\left(  k\right)  }\frac{z^{n-k}}{n!}\\
&  =%
{\displaystyle\sum\limits_{n\geq0}}
\left(
{\displaystyle\sum\limits_{k=0}^{n+p}}
\binom{n+p}{k}\binom{n+p}{p}^{-1}\phi_{n+p-k}B_{k}^{\left(  p\right)
}\right)  \frac{z^{n}}{n!}\\
&  -%
{\displaystyle\sum\limits_{n\geq0}}
\left(
{\displaystyle\sum\limits_{k=1}^{p}}
\binom{n+k}{k}^{-1}\binom{p}{k}B_{n+k}^{\left(  k\right)  }\right)
\frac{z^{n}}{n!}.
\end{align*}
Comparing the coefficients of $\frac{z^{n}}{n!},$ on both sides, we reach the result (\ref{A22}).
\end{proof}
The following lemma will be useful for the proof of next theorem.
\begin{lemma}
For $p\geq0$, we have%
\[
\text{ }_{1}F_{1}\left(
\begin{array}
[c]{c}%
p+1\\
p+2
\end{array}
;1-e^{z}\right)  =\left(  -1\right)  ^{p}\left(  p+1\right)  \left(
e^{-z}\frac{d}{dz}\right)  ^{p}\text{ }_{1}F_{1}\left(
\begin{array}
[c]{c}%
1\\
2
\end{array}
;1-e^{z}\right)
\]
and
\[
_{1}F_{1}\left(
\begin{array}
[c]{c}%
1\\
2
\end{array}
;1-e^{z}\right)  =\left(  \frac{1-\exp\left(  1-e^{z}\right)  }{e^{z}%
-1}\right)  .
\]
\end{lemma}
\begin{proof}
By induction on $p$.
\end{proof}
\begin{theorem}
\label{AA2}For $p\geq1$, we have%
\[%
{\displaystyle\sum\limits_{n\geq0}}
\mathcal{B}_{n,p}\frac{z^{n}}{n!}=\exp\left(  e^{z}\text{ }-1\right)  \left(
-1\right)  ^{p-1}p\left(  e^{-z}\frac{d}{dz}\right)  ^{p-1}\text{ }\left(
\frac{1-\exp\left(  1-e^{z}\right)  }{e^{z}-1}\right)  .
\]
\end{theorem}
\begin{proof}
According to the Kummer transformation \cite{Askey}, we can write%
\[
_{1}F_{1}\left(
\begin{array}
[c]{c}%
1\\
p+1
\end{array}
;e^{z}-1\right)  =\exp\left(  e^{z}\text{ }-1\right)  \text{\ }_{1}%
F_{1}\left(
\begin{array}
[c]{c}%
p\\
p+1
\end{array}
;1-e^{z}\right)
\]
By above Lemma we obtain the result.
\end{proof}
The next result give us the integral representation of Ces\`{a}ro type
\cite{Cesaro}.
\begin{theorem}
For $n\geq1$, we have%
\begin{equation}
B_{n,p}=\frac{2n!p!}{\pi e}\operatorname{Im}\int_{0}^{\pi}\left(  \frac
{\exp(\left(  \exp\left(  e^{i\theta}\right)  \right)  )}{\left(  \exp\left(
e^{i\theta}\right)  -1\right)  ^{p}}-e%
{\displaystyle\sum\limits_{l=0}^{p-1}}
\frac{\left(  \exp\left(  e^{i\theta}\right)  -1\right)  ^{l-p}}{l!}\right)
\sin\left(  n\theta\right)  d\theta, \label{DIF}%
\end{equation}
where $\operatorname{Im}\left(  z\right)  $, denotes the imaginary part of $z$.
\end{theorem}
\begin{proof}
By using the integral identity:%
\[%
\genfrac{\{}{\}}{0pt}{}{n}{k}%
=\frac{2n!}{\pi}\operatorname{Im}\int_{0}^{\pi}\frac{\left(  \exp\left(
e^{i\theta}\right)  -1\right)  ^{k}}{k!}\sin\left(  n\theta\right)  d\theta
\]
and (\ref{Exp11}) we get (\ref{DIF}).
\end{proof}
\subsection{Recurrence relations}
Now we give a recurrence relation for the $p$-Bell numbers $\mathcal{B}%
_{n,p}.$ The proof is based on (\ref{R2}).

\begin{theorem}
The $p$-Bell numbers satisfies the following recurrence relation%
\begin{equation}
\mathcal{B}_{n+1,p}=\left(  n+1\right)  \mathcal{B}_{n,p}-%
{\displaystyle\sum\limits_{k=0}^{n-2}}
\binom{n}{k}\left(  -1\right)  ^{n-k}\mathcal{B}_{k+1,p}-\frac{p}%
{p+1}\mathcal{B}_{n,p+1} \label{R3}%
\end{equation}
with the initial sequence $\mathcal{B}_{0,p}=1.$
\end{theorem}
\begin{proof}
By differentiation (\ref{R2}) with respect to $z$, we obtain%
\begin{align*}
e^{-z}\frac{d}{dz}f_{p}\left(  z\right)   &  =p%
{\displaystyle\int\limits_{0}^{1}}
t\exp(\left(  e^{z}-1\right)  t)\left(  1-t\right)  ^{p-1}dt\\
&  =-p%
{\displaystyle\int\limits_{0}^{1}}
\left(  1-t-1\right)  \exp(\left(  e^{z}-1\right)  t)\left(  1-t\right)
^{p-1}dt\\
&  =-p%
{\displaystyle\int\limits_{0}^{1}}
\exp(\left(  e^{z}-1\right)  t)\left(  1-t\right)  ^{p}dt+p%
{\displaystyle\int\limits_{0}^{1}}
\exp(\left(  e^{z}-1\right)  t)\left(  1-t\right)  ^{p-1}dt\\
&  =-\frac{p}{p+1}f_{p+1}\left(  z\right)  +f_{p}\left(  z\right),
\end{align*}
or equivalently%
\[
e^{-z}%
{\displaystyle\sum\limits_{n\geq0}}
\mathcal{B}_{n+1,p}\frac{z^{n}}{n!}=-\frac{p}{p+1}%
{\displaystyle\sum\limits_{n\geq0}}
\mathcal{B}_{n,p+1}\frac{z^{n}}{n!}+%
{\displaystyle\sum\limits_{n\geq0}}
\mathcal{B}_{n,p}\frac{z^{n}}{n!}.
\]
After some rearrangement, we get%
\[%
{\displaystyle\sum\limits_{n\geq0}}
\left(
{\displaystyle\sum\limits_{k=0}^{n}}
\binom{n}{k}\left(  -1\right)  ^{n-k}\mathcal{B}_{k+1,p}\right)  \frac{z^{n}%
}{n!}=-\frac{p}{p+1}%
{\displaystyle\sum\limits_{n\geq0}}
\mathcal{B}_{n,p+1}\frac{z^{n}}{n!}+%
{\displaystyle\sum\limits_{n\geq0}}
\mathcal{B}_{n,p}\frac{z^{n}}{n!}.
\]
The conclusion follows by comparing coefficients of $\frac{z^{n}}{n!}.$
\end{proof}
By using (\ref{R3}), we list the $p$-Bell numbers $\mathcal{B}_{n,p}$
for $0\leq n\leq6$ and $0\leq p\leq3$
\[
\mathcal{S}_{p}=%
\begin{pmatrix}
1 & 1 & 1 & 1 & \cdots & \mathcal{B}_{0,p}\\
1 & 1/2 & 1/3 & 1/4 & \cdots & \mathcal{B}_{1,p}\\
2 & 5/6 & 1/2 & 7/20 & \cdots & \mathcal{B}_{2,p}\\
5 & 7/4 & 14/15 & 3/5 & \cdots & \mathcal{B}_{3,p}\\
15 & 68/15 & 13/6 & 179/140 & \cdots & \mathcal{B}_{4,p}\\
52 & 167/12 & 127/21 & 185/56 & \cdots & \mathcal{B}_{5,p}\\
203 & 2057/42 & 235/12 & 8389/840 & \cdots & \mathcal{B}_{6,p}\\
\vdots & \vdots & \vdots & \vdots &  & \\
\mathcal{B}_{n,0} & \mathcal{B}_{n,1} & \mathcal{B}_{n,2} & \mathcal{B}_{n,3}
&  &
\end{pmatrix}
\]
When $p=0,$ the formula (\ref{R3}) becomes
\begin{corollary}
We have%
\[
\phi_{n+1}=\left(  n+1\right)  \phi_{n}+%
{\displaystyle\sum\limits_{k=1}^{n-1}}
\left(  -1\right)  ^{n-k}\binom{n}{k-1}\phi_{k}\ .
\]
\end{corollary}
The next result presents a different types of recurrence for
$p$-Bell numbers $\mathcal{B}_{n,p}$. We need the following lemma in order to
prove the Theorem \ref{AAA}.
\begin{lemma}
We have%
\begin{equation}
_{1}F_{1}\left(
\begin{array}
[c]{c}%
1\\
p+1
\end{array}
;z\right)  =\left(  1+\frac{z}{p+1}\right)  \text{ }_{1}F_{1}\left(
\begin{array}
[c]{c}%
1\\
p+2
\end{array}
;z\right)  -\frac{z}{p+2}\text{ }_{1}F_{1}\left(
\begin{array}
[c]{c}%
1\\
p+3
\end{array}
;z\right)  . \label{C1}%
\end{equation}
\end{lemma}
\begin{proof}
This comes directly from $\left(  \ref{R5}\right)  .$
\end{proof}
\begin{theorem}
\label{AAA}The $p$-Bell numbers satisfy the following recurrence relation%
\[
\mathcal{B}_{n+1,p+1}=\mathcal{B}_{n+1,p}-\frac{n+1}{p+1}\mathcal{B}%
_{n,p+1}+\frac{n+1}{p+2}\mathcal{B}_{n,p+2}%
\]
with the initial sequence $\mathcal{B}_{0,p}=1$ and the final sequence
$\mathcal{B}_{n,0}=\phi_{n}$\ .
\end{theorem}
\begin{proof}
\bigskip By $\left(  \ref{C1}\right)  $ with $z:=e^{z}-1$, we have%
\begin{align*}
f_{p}\left(  z\right)   &  =\left(  1+\frac{z}{p+1}\right)  f_{p+1}\left(
z\right)  -\frac{z}{p+2}f_{p+2}\left(  z\right) \\%
{\displaystyle\sum\limits_{n\geq0}}
\mathcal{B}_{n,p}\frac{z^{n}}{n!}  &  =%
{\displaystyle\sum\limits_{n\geq0}}
\left(  \mathcal{B}_{n,p+1}+\frac{1}{p+1}n\mathcal{B}_{n-1,p+1}-\frac{1}%
{p+2}n\mathcal{B}_{n-1,p+2}\right)  \frac{z^{n}}{n!}.
\end{align*}
Equating the coefficients of $\frac{z^{n}}{n!}$ and using (\ref{R3}), we get
the result.
\end{proof}
In the next, we propose an algorithm, which is based on a three-term
recurrence relation, for calculating the $p$-Bell numbers. Let consider the
following sequence $\mathcal{Z}_{n,m}\left(  p\right)  $ with two indices
by%
\begin{equation}
\mathcal{Z}_{n,m}\left(  p\right)  =\dbinom{m+p}{p}{\sum\limits_{k=0}^{m}%
}s\left(  m,k\right)  \mathcal{B}_{n+k,p}. \label{DA1}%
\end{equation}
From Theorem \ref{DA2}, we can write
\[
\mathcal{Z}_{n,m}\left(  p\right)  =\dbinom{m+p}{p}{\sum\limits_{k=0}^{n}}%
\genfrac{\{}{\}}{0pt}{}{n+m}{k+m}%
_{m}\dbinom{m+k+p}{p}^{-1},
\]
with%
\[
\mathcal{Z}_{0,m}\left(  p\right)  =1\text{ and }\mathcal{Z}_{n,0}\left(
p\right)  =\mathcal{B}_{n,p}.
\]
\begin{theorem}
\label{Zth}The $\mathcal{Z}_{n,m}\left(  p\right)  $ satisfy the following
three-term recurrence relation:%
\begin{equation}
\mathcal{Z}_{n+1,m}\left(  p\right)  =\frac{m+1}{m+p+1}\mathcal{Z}%
_{n,m+1}\left(  p\right)  +m\mathcal{Z}_{n,m}\left(  p\right)  , \label{DT5}%
\end{equation}
with the initial sequence given by $\mathcal{Z}_{0,m}\left(  p\right)  =1.$
\end{theorem}
\begin{proof}
From (\ref{DT4}) and (\ref{DA1}), we have
\begin{align*}
\mathcal{Z}_{n,m+1}\left(  p\right)   &  =\dbinom{m+p+1}{p}{\sum
\limits_{k=0}^{m+1}}s\left(  m+1,k\right)  \mathcal{B}_{n+k,p}\\
&  =\dbinom{m+p+1}{p}{\sum\limits_{k=0}^{m+1}}\left(  \left(  s\left(
m,k-1\right)  -ms\left(  m,k\right)  \right)  \right)  \mathcal{B}_{n+k,p}\\
&  =\dbinom{m+p+1}{p}{\sum\limits_{k=0}^{m}}s\left(  m,k\right)
\mathcal{B}_{n+k+1,p}-\dbinom{m+p+1}{p}m{\sum\limits_{k=0}^{m}}s\left(
m,k\right)  \mathcal{B}_{n+k,p}\\
&  =\frac{m+p+1}{m+1}\mathcal{Z}_{n+1,m}\left(  p\right)  -m\frac{m+p+1}%
{m+1}\mathcal{Z}_{n,m}\left(  p\right)  ,
\end{align*}
which is obviously equivalent to (\ref{DT5}).
\end{proof}
\begin{remark}
Note that for $p=0$, the recurrence relation (\ref{DT5}) reduces to a known
formula in \cite{Rahmani2014}.
\end{remark}
\section{Poly-Bell numbers}
It is well known that the Cauchy formula for repeated integration \cite{Podlubny} is given by%
\begin{equation}
\mathcal{D}^{-n}f\left(  x\right)  =\frac{1}{\left(  n-1\right)  !}%
{\displaystyle\int\limits_{0}^{x}}
\left(  x-t\right)  ^{n-1}f\left(  t\right)  dt, \label{Cau}%
\end{equation}
where $\mathcal{D}^{-n}$ denotes the $n$-fold integral and $f$ be a continuous
function on the real line. Comparing the formula (\ref{Cau}) with the formula
(\ref{za}) we obtain:
\begin{theorem}
For $p\geq1$, we have%
\begin{align}
\mathcal{B}_{n,p}  &  =p!\mathcal{D}^{-p}\left.  \phi_{n}\left(  x\right)
\right\vert _{x=1}\nonumber\\
&  =p!%
{\displaystyle\int\limits_{0}^{1}}
dx_{p}%
{\displaystyle\int\limits_{0}^{x_{p}}}
dx_{p-1}\cdots%
{\displaystyle\int\limits_{0}^{x_{2}}}
\phi_{n}\left(  x_{1}\right)  dx_{1}. \label{itter}%
\end{align}
\end{theorem}
The iterated integral expression (\ref{itter}) motivate a definition of
poly-Bell numbers. Define the poly-Bell numbers $\mathbb{B}_{n}^{\left(
p\right)  }$as follows:%
\[
\mathbb{B}_{n}^{\left(  p\right)  }=\frac{\mathcal{B}_{n,p}}{p!}.
\]
For $\mathbb{B}_{n}^{\left(  p\right)  }$ the generating function is given by
\[%
{\displaystyle\sum\limits_{n\geq0}}
\mathbb{B}_{n}^{\left(  p\right)  }\frac{z^{n}}{n!}=%
{\displaystyle\sum\limits_{n\geq0}}
\frac{\left(  e^{z}-1\right)  ^{n}}{\Gamma\left(  n+p+1\right)  }%
\]
and can be extended for $p$ negative as%
\begin{equation}%
{\displaystyle\sum\limits_{n\geq0}}
\mathbb{B}_{n}^{\left(  -p\right)  }\frac{z^{n}}{n!}=%
{\displaystyle\sum\limits_{n\geq p}}
\frac{\left(  e^{z}-1\right)  ^{n}}{\Gamma\left(  n-p+1\right)  }.
\label{pol-bel}%
\end{equation}
An explicit formula of the poly-Bell numbers at a negative upper index $p$ can
be obtained by a direct use of (\ref{pol-bel}).
\begin{theorem}
The poly-Bell numbers $\mathbb{B}_{n}^{\left(  -p\right)  }$is positive and we
have%
\begin{equation}
\mathbb{B}_{n}^{\left(  -p\right)  }=%
{\displaystyle\sum\limits_{k=p}^{n}}
\frac{k!}{\left(  k-p\right)  !}%
\genfrac{\{}{\}}{0pt}{}{n}{k}%
. \label{Aic2}%
\end{equation}
\end{theorem}
\begin{corollary}
For $p\geq0$, we have%
\begin{align*}
\mathbb{B}_{n}^{\left(  -p\right)  }  &  =\frac{d^{p}}{dx^{p}}\left.  \phi
_{n}\left(  x\right)  \right\vert _{x=1}\\
&  =p!%
{\displaystyle\sum\limits_{j=p}^{n}}
\binom{n}{j}%
\genfrac{\{}{\}}{0pt}{}{j}{p}%
\phi_{n-j}.
\end{align*}
\end{corollary}
\begin{proof}
From (\ref{Aich3}), we have%
\begin{align*}
\frac{d^{p}}{dx^{p}}\phi_{n}\left(  x\right)   &  =%
{\displaystyle\sum\limits_{k=0}^{n}}
\genfrac{\{}{\}}{0pt}{}{n}{k}%
\frac{d^{p}}{dx^{p}}x^{k}\\
&  =p!%
{\displaystyle\sum\limits_{k=p}^{n}}
\binom{k}{p}%
\genfrac{\{}{\}}{0pt}{}{n}{k}%
x^{k-p}%
\end{align*}
which reduces to $\mathbb{B}_{n}^{\left(  -p\right)  }$ by setting $x=1$. On
the other hand, from (\ref{Bellgen}), we have%
\begin{align*}%
{\displaystyle\sum\limits_{n\geq0}}
\frac{d^{p}}{dx^{p}}\phi_{n}\left(  x\right)  \frac{z^{n}}{n!}  &
=\frac{d^{p}}{dx^{p}}\exp\left(  x\left(  e^{z}-1\right)  \right) \\
&  =\left(  e^{z}-1\right)  ^{p}\exp\left(  x\left(  e^{z}-1\right)  \right)
\\
&  =p!%
{\displaystyle\sum\limits_{n\geq p}}
\genfrac{\{}{\}}{0pt}{}{n}{p}%
\frac{z^{p}}{p!}%
{\displaystyle\sum\limits_{n\geq0}}
\phi_{n}\left(  x\right)  \frac{z^{n}}{n!}\\
&  =p!%
{\displaystyle\sum\limits_{n\geq0}}
\left(
{\displaystyle\sum\limits_{j=0}^{n}}
\binom{n}{j}%
\genfrac{\{}{\}}{0pt}{}{j}{p}%
\phi_{n-j}\left(  x\right)  \right)  \frac{z^{n}}{n!}.
\end{align*}
Equating coefficients of $\frac{z^{n}}{n!}$ and setting $x=1,$ we get the result.
\end{proof}
A natural question that arises is: the poly-Bernoulli numbers have the duality
property \cite{Kaneko1,Kaneko2}. What about the poly-Bell numbers?
\begin{theorem}
\label{aich}The double generating function of $\mathbb{B}_{n}^{\left(
-p\right)  }$ is given by%
\[%
{\displaystyle\sum\limits_{p\geq0}}
{\displaystyle\sum\limits_{n\geq p}}
\mathbb{B}_{n}^{\left(  -p\right)  }\frac{z^{n}}{n!}\frac{y^{p}}{p!}%
=\exp\left(  \left(  y+1\right)  \left(  e^{z}-1\right)  \right)  .
\]
\end{theorem}
By the above theorem, we can answer negatively the question, the duality
property does not hold for the poly-Bell numbers.
\begin{proof}
The double generating function of $\mathbb{B}_{n}^{\left(  -p\right)  }$ can
be obtained by using (\ref{pol-bel}) and (\ref{secondkind})%
\begin{align*}%
{\displaystyle\sum\limits_{p\geq0}}
{\displaystyle\sum\limits_{n\geq p}}
\mathbb{B}_{n}^{\left(  -p\right)  }\frac{z^{n}}{n!}\frac{y^{p}}{p!}  &  =%
{\displaystyle\sum\limits_{p\geq0}}
{\displaystyle\sum\limits_{n\geq p}}
\frac{\left(  e^{z}-1\right)  ^{n}}{\left(  n-p\right)  !}\frac{y^{p}}{p!}\\
&  =%
{\displaystyle\sum\limits_{p\geq0}}
\left(  y\left(  e^{z}-1\right)  \right)  ^{p}\frac{1}{p!}%
{\displaystyle\sum\limits_{i\geq0}}
\frac{\left(  e^{z}-1\right)  ^{i}}{i!}\\
&  =\exp\left(  \left(  y+1\right)  \left(  e^{z}-1\right)  \right)  .
\end{align*}
\end{proof}
As a consequence of Theorem \ref{aich}, we obtain the generating function of
poly-Bell numbers $\mathbb{B}_{n}^{\left(  -p\right)  }$ at negative upper
index $p$.
\begin{corollary}
We have
\[%
{\displaystyle\sum\limits_{p\geq0}}
\mathbb{B}_{n}^{\left(  -p\right)  }\frac{y^{p}}{p!}=\phi_{n}\left(
1+y\right)  .
\]
\end{corollary}
\begin{proof}
It follows immediately from%
\[
\exp\left(  \left(  y+1\right)  \left(  e^{z}-1\right)  \right)  =%
{\displaystyle\sum\limits_{n\geq0}}
\phi_{n}\left(  1+y\right)  \frac{z^{n}}{n!}.
\]
\end{proof}
Here is a small table of the numbers $\mathbb{B}_{n}^{\left(  p\right)
}\left(  -4\leq p\leq-1,0\leq n\leq9\right)  .$%

\[
\begin{tabular}
[c]{c|cccccccccc}\hline
$p\backslash n$ & $0$ & $1$ & $2$ & $3$ & $4$ & $5$ & $6$ & $7$ & $8$ &
$9$\\\hline
$-1$ & $0$ & $1$ & $3$ & $10$ & $37$ & $151$ & $674$ & $3263$ & $17007$ &
$94828$\\
$-2$ & $0$ & $0$ & $2$ & $12$ & $62$ & $320$ & $1712$ & $9604$ & $56674$ &
$351792$\\
$-3$ & $0$ & $0$ & $0$ & $6$ & $60$ & $450$ & $3120$ & $21336$ & $147756$ &
$1048830$\\
$-4$ & $0$ & $0$ & $0$ & $0$ & $24$ & $360$ & $3720$ & $33600$ & $287784$ &
$2424744$\\\hline
\end{tabular}
\]

\section{The $p$-Bell polynomials}
For $p\geq0,$ let us consider the $p$-Bell polynomials $\mathcal{B}_{n,p}\left(  x\right)  $ is defined by means of the following generating
function%
\begin{equation}
F_{p}\left(  x\right)  :=%
{\displaystyle\sum\limits_{n\geq0}}
\mathcal{B}_{n,p}\left(  x\right)  \frac{z^{n}}{n!}=\text{ }_{1}F_{1}\left(
\begin{array}
[c]{c}%
1\\
p+1
\end{array}
;e^{z}-1\right)  e^{xz}. \label{definition}%
\end{equation}
Now, using (\ref{definition}) we can easily get the explicit formula for the
$p$-Bell polynomials.
\begin{equation}
\mathcal{B}_{n,p}\left(  x\right)  =%
{\displaystyle\sum\limits_{k=0}^{n}}
\binom{n}{k}\mathcal{B}_{k,p}x^{n-k}, \label{F2}%
\end{equation}
For every integer $p,$ we define also the following class of a polynomials
$\mathbb{B}_{n}^{\left(  p\right)  }\left(  x\right)  ,$ which we call
poly-Bell polynomials, by%
\[
\mathbb{B}_{n}^{\left(  p\right)  }\left(  x\right)  =\frac{\mathcal{B}%
_{n,p}(x)}{p!}.
\]
The first few $p$-Bell polynomials are%
\begin{align*}
\mathcal{B}_{0,p}\left(  x\right)   &  =1,\\
\mathcal{B}_{1,p}\left(  x\right)   &  =x+\frac{1}{p+1},\\
\mathcal{B}_{2,p}\left(  x\right)   &  =x^{2}+\frac{2x}{p+1}+\frac
{p+4}{\left(  p+1\right)  \left(  p+2\right)  },\\
\mathcal{B}_{3,p}\left(  x\right)   &  =x^{3}+\frac{3x^{2}}{p+1}%
+\frac{3\left(  p+4\right)  x}{\left(  p+1\right)  \left(  p+2\right)  }%
+\frac{p^{2}+11p+30}{\left(  p+1\right)  \left(  p+2\right)  \left(
p+3\right)  }%
\end{align*}
and
\begin{multline*}
\mathcal{B}_{4,p}\left(  x\right)  =x^{4}+\frac{4x^{3}}{p+1}+\frac{6\left(
p+4\right)  x^{2}}{\left(  p+1\right)  \left(  p+2\right)  }+\frac{4\left(
p^{2}+11p+30\right)  x}{\left(  p+1\right)  \left(  p+2\right)  \left(
p+3\right)  }\\
+\frac{p^{3}+23p^{2}+160p+360}{\left(  p+1\right)  \left(  p+2\right)  \left(
p+3\right)  \left(  p+4\right)  }.
\end{multline*}
We present here another expression for the explicit formula of $\mathcal{B}%
_{n,p}\left(  x\right)  $ involving weighted Stirling numbers $S_{n}%
^{k}\left(  x\right)  $ of the second kind \cite{Carlitz1}. Recall that $S_{n}^{k}\left(
x\right)  $ are defined by%
\begin{align*}
S_{n}^{k}\left(  x\right)   &  =\frac{1}{k!}\Delta^{k}x^{n}\\
&  =%
{\displaystyle\sum\limits_{i=0}^{n}}
\binom{n}{i}%
\genfrac{\{}{\}}{0pt}{}{i}{k}%
x^{n-i},
\end{align*}
where $\Delta$ denotes the forward difference operator. These numbers are
related to $r$-Stirling numbers of the second kind$%
\genfrac{\{}{\}}{0pt}{}{n}{k}%
_{r}$ and Whitney numbers of the second kind \cite{Rahmani2014a} $W_{m,r}\left(  n,k\right)  $\ by%
\[
S_{n}^{k}\left(  r\right)  =%
\genfrac{\{}{\}}{0pt}{}{n+r}{k+r}%
_{r}\text{ and }S_{n}^{k}\left(  \frac{r}{m}\right)  =\frac{1}{m^{n-k}}%
W_{m,r}\left(  n,k\right)  ,
\]
respectively.
\begin{theorem}
For $p\geq0$, we have%
\begin{equation}
\mathcal{B}_{n,p}\left(  x\right)  =%
{\displaystyle\sum\limits_{k=0}^{n}}
\binom{k+p}{k}^{-1}S_{n}^{k}\left(  x\right)  . \label{expw}%
\end{equation}
\end{theorem}
\begin{proof}
From (\ref{expw}) and the generating function of $S_{n}^{k}\left(  x\right)  $%
\[
\frac{1}{k!}e^{xz}\left(  e^{z}-1\right)  ^{k}=%
{\displaystyle\sum\limits_{n\geq k}}
S_{n}^{k}\left(  x\right)  \frac{z^{n}}{n!},
\]
we have%
\begin{align*}%
{\displaystyle\sum\limits_{n\geq0}}
\left(
{\displaystyle\sum\limits_{k=0}^{n}}
\binom{k+p}{k}^{-1}S_{n}^{k}\left(  x\right)  \right)  \frac{z^{n}}{n!}  &  =%
{\displaystyle\sum\limits_{k=0}^{\infty}}
\binom{k+p}{k}^{-1}%
{\displaystyle\sum\limits_{n\geq0}}
S_{n}^{k}\left(  x\right)  \frac{z^{n}}{n!}\\
&  =%
{\displaystyle\sum\limits_{k=0}^{\infty}}
\binom{k+p}{k}^{-1}\frac{1}{k!}e^{xz}\left(  e^{z}-1\right)  ^{k}\\
&  =e^{xz}%
{\displaystyle\sum\limits_{k=0}^{\infty}}
\frac{p!}{\left(  k+p\right)  !}\left(  e^{z}-1\right)  ^{k}\\
&  =e^{xz}%
{\displaystyle\sum\limits_{n\geq0}}
\mathcal{B}_{n,p}\frac{z^{n}}{n!}\\
&  =%
{\displaystyle\sum\limits_{n\geq0}}
\mathcal{B}_{n,p}\left(  x\right)  \frac{z^{n}}{n!}.
\end{align*}
This evidently completes the proof of the theorem.
\end{proof}
In particular, we have%
\[
\mathcal{B}_{n,p}\left(  r\right)  =%
{\displaystyle\sum\limits_{k=0}^{n}}
\binom{k+p}{k}^{-1}%
\genfrac{\{}{\}}{0pt}{}{n+r}{k+r}%
_{r}%
\]
and
\[
\mathcal{B}_{n,p}\left(  \frac{r}{m}\right)  =%
{\displaystyle\sum\limits_{k=0}^{n}}
\frac{1}{m^{n-k}}\binom{k+p}{k}^{-1}W_{m,r}\left(  n,k\right)  .
\]
Now, we want to generalize Theorem \ref{Zth} to the polynomials case. Let
consider the polynomials $\mathcal{Z}_{n,m}\left(  z;p\right)  $ which is defined by%
\begin{equation}
\mathcal{Z}_{n,m}\left(  x;p\right)  =%
{\displaystyle\sum\limits_{k=0}^{n}}
\binom{n}{k}\mathcal{Z}_{k,m}\left(  p\right)  x^{n-k} \label{TDD5}%
\end{equation}
with $\mathcal{Z}_{0,m}\left(  x;p\right)  =1$ and $\mathcal{Z}_{n,0}\left(
x;p\right)  =\mathcal{B}_{n,p}\left(  x\right)  .$
\begin{theorem}
The polynomials $\mathcal{Z}_{n,m}\left(  x;p\right)  $ satisfy the following
three-term recurrence relation:%
\begin{equation}
\mathcal{Z}_{n+1,m}\left(  x;p\right)  =\frac{m+1}{m+p+1}\mathcal{Z}%
_{n,m+1}\left(  x;p\right)  +\left(  m+x\right)  \mathcal{Z}_{n,m}\left(
x;p\right)  \label{polrec2}%
\end{equation}
with the initial sequence given by
\[
\mathcal{Z}_{0,m}\left(  x;p\right)  =1.
\]
\end{theorem}
\begin{proof}
From (\ref{TDD5}) and (\ref{DT5}), we have%
\begin{multline*}
x\frac{d}{dx}\mathcal{Z}_{n,m}\left(  x;p\right)  =n\mathcal{Z}_{n,m}\left(
x;p\right)  -\frac{n\left(  m+1\right)  }{m+p+1}%
{\displaystyle\sum\limits_{k=0}^{n-1}}
\binom{n-1}{k}\mathcal{Z}_{k,m+1}\left(  x;p\right)  x^{n-k-1}\\
-nm%
{\displaystyle\sum\limits_{k=0}^{n-1}}
\binom{n-1}{k}\mathcal{Z}_{k,m}\left(  x;p\right)  x^{n-k-1}.
\end{multline*}
Thus%
\[
x\mathcal{Z}_{n-1,m}\left(  x;p\right)  =\mathcal{Z}_{n,m}\left(  x;p\right)
-\frac{m+1}{m+p+1}\mathcal{Z}_{n-1,m+1}\left(  x;p\right)  -m\mathcal{Z}%
_{n-1,m}\left(  x;p\right)  ,
\]
which is equivalent to (\ref{polrec2}).
\end{proof}

The next theorem establishes the integral representation for $p$-Bell
polynomials. From (\ref{definition}) and (\ref{R2}), we have%
\[
F_{p}\left(  x\right)  =p%
{\displaystyle\int\limits_{0}^{1}}
\left(  1-t\right)  ^{p-1}\exp(xz+t\left(  e^{z}-1\right)  )dt.
\]
Since%
\begin{align*}
\exp(xz+t\left(  e^{z}-1\right)  )  &  =%
{\displaystyle\sum\limits_{n\geq0}}
G_{n,x}\left(  t\right)  \frac{z^{n}}{n!}\\
&  =%
{\displaystyle\sum\limits_{n\geq0}}
\left(
{\displaystyle\sum\limits_{k=0}^{n}}
\binom{n}{k}\phi_{k}\left(  t\right)  x^{n-k}\right)  \frac{z^{n}}{n!},
\end{align*}
we get the following theorem.
\begin{theorem}
The $p$-Bell polynomials $\mathcal{B}_{n,p}\left(  x\right)  $ satisfy the
following integral representation%
\begin{equation}
\mathcal{B}_{n,p}\left(  x\right)  =p%
{\displaystyle\int\limits_{0}^{1}}
\left(  1-t\right)  ^{p-1}G_{n,x}\left(  t\right)  dt. \label{ZE}%
\end{equation}
\end{theorem}
Note that the polynomials $G_{n,x}\left(  t\right)  $ are a particular case of
the more general polynomials considered by Corcino et al. in \cite{Corcino}.
Now using the properties of $G_{n,x}\left(  t\right)  ,$ we get the following corollaries.
\begin{corollary}
[Dobinski's formula]For $p\geq1$, we have%
\[
\mathcal{B}_{n,p}\left(  x\right)  =p%
{\displaystyle\sum\limits_{k\geq0}}
\frac{\left(  x+k\right)  ^{n}}{k!}\text{ }_{1}F_{1}\left(
\begin{array}
[c]{c}%
k+1\\
p+k+1
\end{array}
;-1\right)  .
\]
\end{corollary}
\begin{proof}
From \cite[Formula (1.18)]{Corcino}, we obtain%
\begin{align*}
\mathcal{B}_{n,p}\left(  x\right)   &  =p%
{\displaystyle\int\limits_{0}^{1}}
\left(  1-t\right)  ^{p-1}\left(  e^{-t}%
{\displaystyle\sum\limits_{k\geq0}}
\frac{t^{k}}{k!}\left(  x+k\right)  ^{n}\right)  dt.\\
&  =p%
{\displaystyle\sum\limits_{k\geq0}}
\frac{\left(  x+k\right)  ^{n}}{k!}%
{\displaystyle\int\limits_{0}^{1}}
e^{-t}t^{k}\left(  1-t\right)  ^{p-1}dt.\\
&  =p%
{\displaystyle\sum\limits_{k\geq0}}
\frac{\left(  x+k\right)  ^{n}}{k!}\text{ }_{1}F_{1}\left(
\begin{array}
[c]{c}%
k+1\\
p+k+1
\end{array}
;-1\right)  .
\end{align*}
\end{proof}
\begin{corollary}
[Recurrence relations ]For $p\geq0$, we have%
\begin{equation}
\mathcal{B}_{n+1,p}\left(  x\right)  =x\mathcal{B}_{n,p}\left(  x\right)  -%
{\displaystyle\sum\limits_{k=0}^{n}}
\binom{n}{k}\left(  \frac{p}{p+1}\mathcal{B}_{k,p+1}\left(  x\right)
-\mathcal{B}_{k,p}\left(  x\right)  \right)  . \label{Cor2}%
\end{equation}
\end{corollary}
\begin{proof}
From \cite[Formula (2.9)]{Corcino}, we obtain%
\begin{align*}
\mathcal{B}_{n+1,p}\left(  x\right)   &  =p%
{\displaystyle\int\limits_{0}^{1}}
\left(  1-t\right)  ^{p-1}\left(  xG_{n,x}\left(  t\right)  +%
{\displaystyle\sum\limits_{k=0}^{n}}
t\binom{n}{k}G_{n-k,x}\left(  t\right)  \right)  dt\\
&  =xp%
{\displaystyle\int\limits_{0}^{1}}
\left(  1-t\right)  ^{p-1}G_{n,x}\left(  t\right)  dt\\
&  -p%
{\displaystyle\sum\limits_{k=0}^{n}}
\binom{n}{k}%
{\displaystyle\int\limits_{0}^{1}}
\left(  \left(  1-t\right)  -1\right)  \left(  1-t\right)  ^{p-1}%
G_{n-k,x}\left(  t\right)  dt.
\end{align*}
After some manipulation we get (\ref{Cor2}).
\end{proof}
Note that when $p=0,$ (\ref{Cor2}) gives%
\[
\phi_{n+1}\left(  x\right)  =x\phi_{n}\left(  x\right)  +%
{\displaystyle\sum\limits_{k=0}^{n}}
\binom{n}{k}\phi_{k}\left(  x\right)  .
\]

\section{Probabilistic representation}

We consider a Poisson distributed random variable $Z$ with intensity
$\lambda>0$. Suppose further that $\lambda$ is a random variable a beta
distribution $\lambda\leadsto\operatorname*{beta}(1,p),$ the probability
density function, for $0\leq t\leq1$, and the parameter $p>0$ as follows:
\begin{align*}
f_{\lambda}(t)  &  =\frac{1}{\beta(1,p)}(1-t)^{p-1}\\
&  =p(1-t)^{p-1},
\end{align*}
because
\[
\beta(1,p)=\frac{1}{p},
\]
where $\beta$ denotes beta function.

Then, the resultant distribution is a particular case of the beta-Poisson
distribution \cite{Grandell,linda}. The probability mass function of $Z$\ is
given by
\begin{align*}
f_{Z}(k)  &  =\int_{%
\mathbb{R}
}f_{(Z,\lambda)}(k,t)dt\\
&  =\int_{%
\mathbb{R}
}f_{Z/\lambda}(k/t)f_{\lambda}(t)dt\\
&  =\int_{0}^{1}p(1-t)^{p-1}\frac{t^{k}e^{-t}}{k!}dt\\
&  =\frac{1}{ek!}\frac{\Gamma(1+p)\Gamma(p+k)}{\Gamma(1+p+k)\Gamma(p)}\text{
}_{1}F_{1}\left(
\begin{array}
[c]{c}%
1\\
p+k+1
\end{array}
;1\right)  ,
\end{align*}
thus, the probability mass function of a weighted Poisson distribution.
The moment generating function of the beta-Poisson distribution is given by
\[
M(t)=\mathbb{E}\left[  e^{tZ}\right]  =\text{ }_{1}F_{1}\left(
\begin{array}
[c]{c}%
p\\
p+1
\end{array}
;e^{t}-1\right)  \text{ },
\]
where $\mathbb{E}$\ is the mathematical expectation.

It is well known that the higher moments $\mathbb{E}\left[  Z^{n}\right]  $ of
the Poisson distribution are $\phi_{n}(t)$. In the next paragraph, we show
that $\mathcal{B}_{n,p}(x)$\ can be viewed as the $n$th moment of a random
variable $x+Z$\ where $Z$\ follows the beta-Poisson law.
\begin{theorem}
Let $Z$ be a random variable follows the beta-Poisson law, we have%
\[
\mathcal{B}_{n,p}(x)=\mathbb{E}\left[  (x+Z)^{n}\right]  .
\]
\end{theorem}
\begin{proof}
From (\ref{ZE}) and \cite[p. 13, formula $\left(  4,17\right)  $]{Corcino}, we
have%
\[
\mathcal{B}_{n,p}(x)=p\int_{0}^{1}(1-t)^{p-1}\mathbb{E}\left[  (x+Z)^{n}%
/\lambda=t\right]  dt,
\]
on the other hand, we have
\[
\mathbb{E}\left[  (x+Z)^{n}\right]  =\mathbb{E}[\mathbb{E}((x+Z)^{n}%
/\lambda=t)],
\]
from which we obtain the result by comparison.
\end{proof}

\end{document}